\documentclass{article}
\usepackage{graphicx}

\usepackage{amsmath, amsthm, amsfonts}
\usepackage{cancel}
\usepackage[normalem]{ulem}
\theoremstyle{plain}
\newtheorem{theorem}{Theorem}[section]
\newtheorem{corollary}{Corollary}[section]

\theoremstyle{definition}
\newtheorem{definition}{Definition}[section]
\newtheorem{observation}{Observation}[section]

\usepackage{float}
\usepackage{calc,tikz}
\usetikzlibrary{arrows,decorations.markings,positioning,shapes.geometric}
\tikzstyle{vertex}=[circle, draw, inner sep=1pt, minimum size=4pt]
\newcommand{\vertex}{\node[vertex]}

\tikzstyle{ann} = [fill=white,font=\footnotesize,inner sep=1pt]
\tikzstyle{arrow} = [thick,<-->,>=stealth]
\date{}

\begin{document}

\title{Characterizations of Some Parity Signed Graphs}

\author{Mukti Acharya$^1$, Joseph Varghese Kureethara$^2$, Thomas Zaslavsky$^3$\\
	\small $^{1,2}$Centre for Mathematical Needs, Christ University, Bengaluru, Karnataka, India;\\
	\small $^1$mukti1948@gmail.com, $^2$frjoseph@christuniversity.in \\
	\small $^3$Department of Mathematical Sciences, Binghamton University,\\ \small Binghamton, New York, U.S.A.; zaslav@math.binghamton.edu\\
}
\date{}

\maketitle

\begin{abstract}
We describe parity labellings of signed graphs; equivalently, cuts of the underlying graph that have nearly equal sides. We characterize the balanced signed graphs which are parity signed graphs. We give structural characterizations of all parity signed stars, bistars, cycles, paths and complete bipartite graphs. The \emph{rna} number of a graph is the smallest cut size that has nearly equal sides; we find it for a few classes of graphs.
\end{abstract}

\section{Introduction}
The concept of signed graph has gained immense popularity in graph theory in the recent decades. Here, we discuss a type of signed graph called a parity signed graph, introduced recently in \cite{Ach}. This is based on the assignment of consecutive positive integers to the vertices of a graph; it is equivalent to a partition of the vertex set of a graph into two subsets, $A$ and $B$, that are as nearly the same size as can be, i.e., such that $|A|-|B|=0,\pm1$. From the standpoint of signed graphs, we wish to know whether a given signed graph is parity signed; we answer that question for signed stars, bistars, cycles, paths and complete bipartite graphs.  We further examine the \emph{rna} number of a graph, which is the size of a smallest cut whose sides are nearly equal, for some types of graph such as stars, wheels, paths, and cycles. (The term \emph{rna} is the Sanskrit word for debt.)

For terminology for graphs we refer to \cite{har69, West} and for signed graphs we refer to \cite{Zas82}. For a detailed conceptual framework in signed graphs, we refer the reader to \cite{Zas05}. All graphs and signed graphs considered here are simple and connected unless mentioned otherwise.

By an $(n,m)$-graph, $G=(V,E)$, we mean a graph $G$ such that $n=|V(G)|$ and $m=|E(G)|$. A signed graph $S=(G, \sigma)$ is a pair of a graph $G=(V,E)$ and a function $\sigma:E(G)\rightarrow \{+,-\}$; edges which receive + and $-$ signs are called positive and negative edges of $S$, respectively. The graph $G$ is called the underlying graph of $S$. By $E^+(S)$ ($E^-(S)$), we denote the set of positive (negative) edges of $S$, so the edge set $E(S)=E^+(S) \cup E^-(S)$. A signed graph is said to be all-positive if $E^-(S)=\emptyset$ and all-negative if $E^+(S)=\emptyset$. While drawing a signed graph, positive edges are drawn as solid line segments and negative edges as dashed line segments, as depicted in Figure \ref{fig:pscycle}. A signed graph is said to be homogeneous if it is either all positive or all negative, and heterogeneous otherwise. By a positive (negative) homogeneous signed graph we mean a signed graph which is all positive (all negative).

Now we give definitions and results which are needed for our work.

\begin{definition}\label{parsign}
	Given a graph $G$ and a bijection $f:V(G) \rightarrow \{1, 2, \dots, n\}$, we define $\sigma_f: E(G) \rightarrow   \{+,-\}$ such that for an edge $uv$ in $G$, $\sigma(uv)=+$ if $f(u)$ and $f(v)$ are of the same parity and $\sigma(uv)=-$ if $f(u)$ and $f(v)$ are of opposite parity.  We define $S_f$ to be the signed graph $(G,\sigma_f)$.
\end{definition}

\begin{definition}\label{par} \cite{Ach}
	A signed graph $S=(G, \sigma)$ is a \textit{parity signed graph}, if there exists a bijection $f:V(G) \rightarrow \{1, 2, \dots, n\}$ such that $\sigma = \sigma_f$.
\end{definition}

\begin{definition}\cite{Ach}\label{rna}
	For a graph $G$, the \emph{rna} (\emph{adhika}) number of $G$, denoted by $\sigma^{-}(G)$ ($\sigma^{+}(G)$), is the cardinality of the smallest $E^{-}(S_f)$ (the largest $E^{+}(S_f)$) under all the possible bijective label assignments $f:V(G) \rightarrow \{1, 2, \dots, n\}$. (The word \emph{adhika} in Sanskrit means excess.)
\end{definition}

The motivation behind the study of parity signed graphs is primarily sociological, as was originally the case of the study of signed graphs itself. Assume that there exist two types of people in an office, divided by their distinct languages. Assume also that only people of the same language get along well. To allot workspace to each employee in that office, the personnel manager makes sure that a minimum discomfort is surfaced due to their proximities. Here, we can model cabins as the vertices of a graph and languages as odd and even positive integers. The \emph{rna} number gives the least level of discomfort in integer terms, i.e., the smaller the \emph{rna} number is, the less the discomfort is.

Further, in any system where there are objects basically binary in nature such as male and female, native and foreigner, positive and negative, good and evil, etc., we can bring in the ideas of parity signed graphs.

\begin{definition}
	In a signed graph $S=(G, \sigma)$, a \emph{positive (negative) section} is a maximal connected sub signed graph of $S$ with all positive (negative) edges. In Figure \ref{img1} we have a signed graph with two positive sections and two negative sections. 
\end{definition}

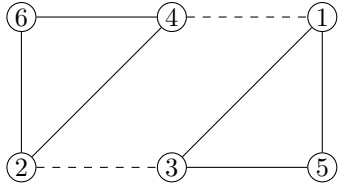
\begin{figure}[h!]\label{fig:pscycle}
	\centering
	\begin{tikzpicture}
	\vertex (u2) at (0,0) {2};
	\vertex (u3) at (2,0)  {3};
	\vertex (u5) at (4,0)  {5};
	\vertex (u1) at (4,2) {1};
	\vertex (u4) at (2,2)  {4};
	\vertex (u6) at (0,2) {6};	
	\path 
	(u1) edge (u3)
	(u3) edge  (u5)
	(u1) edge (u5)
	(u6) edge  (u4)
	(u2) edge (u6)
	(u2) edge (u4)
	
	(u2) edge [dashed] (u3)
	
	(u4) edge [dashed] (u1)
	;
	\end{tikzpicture}
	\caption{A parity signed graph with two negative sections and two positive sections.}\label{img1}
\end{figure}

If one is interested in the labels themselves, the following observations give ways to generate a new parity labelling from an existing parity labelling in a given parity signed graph.
\begin{observation}
	Let $f(v_i)=i$ be the labels of a parity signed graph $S$ with $n$ vertices.  Then $f(v_i)=n+1-i$ is another parity labelling of $S$. 
\end{observation}

\begin{observation}
	Let $O$ ($E$) be the set of odd (even) labelled vertices of a parity signed graph. Then, every permutation of labels on $O$ ($E$) gives a parity labelling for the parity signed graph. 
\end{observation}

\section{Characterization of Parity Signed Graphs}

One important exploration in the study of signed graphs is about their balanced nature. Harary introduced this idea in \cite{har}. A signed graph is balanced if every cycle in it has an even number of negative edges.
\begin{theorem}\label{thm_balanced}
	Every parity signed graph is balanced.
\end{theorem}

\begin{proof}
	It has been shown in \cite{Ach} that every parity signed cycle has an even number of negative edges. Hence, every parity signed graph is balanced.
\end{proof}

In this perspective, the next result is very important.

\begin{theorem}\label{char}
	A signed graph $S$ is a parity signed graph if and only if its vertex set V(S) can be partitioned  into two subsets $V_1(S)$ and $V_2(S)$ such that negative edges lie across $V_1(S)$ and $V_2(S)$ and $||V_1(S)|-|V_2(S)||\le 1$.
\end{theorem}
Since we assume graphs are connected, the partition is uniquely determined.
\begin{proof}
	For necessity, assume that $S$ is a parity signed graph. Now, partition the vertex set $V(S)$ into two subsets $V_1(S)$ and $V_2(S)$ such that vertices of  $V_1(S)$ and $V_2(S)$ are labelled with even and odd labels, respectively. Clearly, negative edges lie across $V_1(S)$ and $V_2(S)$.
	
	Now, if $n=|V(S)|$, is a an even number, then $|V_1(S)|=|V_2(S)|$.
	If $n=|V(S)|$, is a an odd number, then $|V_1(S)|+1=|V_2(S)|$.
	Hence, $||V_1(S)|-|V_2(S)||\le 1$.
	
	Sufficiency is easy to see.
\end{proof}

We provide an elementary algorithm to determine whether a (connected) signed graph is a parity signed graph.

\indent \textbf{Step 1}. Decide whether $S$ is balanced.  If not, ``No''.\\
\indent \textbf{Step 2}. Contract each positive section to a vertex and label the new vertex with the order of that section.\\
\indent \textbf{Step 3}. Any unlabeled vertex gets the label 1.\\
\indent \textbf{Step 4}. The resulting graph is negative homogeneous and balanced, so it is bipartite.  Find the two vertex classes and sum up the vertex labels in each class.  If the sums differ by more than 1, then ``No''.  Otherwise, ``Yes''.\\

Given a balanced signed graph, we wish to determine whether it is parity signed.  We can answer this question in some cases.  We begin with cycles.

A section in a cycle is a path, unless it is the whole cycle.  It is called odd or even if its length is odd or even, respectively.  We write $l(P)$ for the length of a path $P$.

\begin{theorem} [Cycle Theorem]\label{thm:cycle}
	Let $C$ be a signed cycle.  Let $C$ have odd negative sections $N_1, N_2, \ldots, N_k$ ($k\geq0$), in cyclic order around $C$.  Let $m_i$ be the number of positive edges between $N_i$ and $N_{i+1}$ for $i=1,2,\ldots,k-1$ and let $m_k$ be the number of positive edges between $N_k$ and $N_1$.  Let $m_o$ be the sum of all $m_i$ for odd $i$ and $m_e$ the sum of all $m_i$ for even $i$.  Then $C$ is a parity signed graph if and only if $k$ is even and either
	\begin{enumerate}
		\item $k=0$ and either $C$ has even length and is all negative, or $C$ has odd length and has exactly one positive edge, or else
		\item $k>0$ and $|m_o-m_e| \leq 1$.
	\end{enumerate}
\end{theorem}

\begin{proof}
	For $C$ to be parity signed, it must be balanced, and then $k$ must be even.  Thus, we assume $C$ is balanced and $k$ is even.  Let $V(C) = H_0 \cup H_1$ be the \emph{Harary bipartition} of $C$, i.e., $H_0 \cap H_1 = \emptyset$ and an edge is positive if and only if its endpoints are both in the same set, $H_0$ or $H_1$.
	
	Consider a negative section of length $l$, $N = v_0v_1\cdots v_l$.  If $v_0 \in H_h$, then all $v_{2j}$ in $N$ belong to $H_h$ and all $v_{2j+1}$ in $N$ belong to $H_{1-h}$.  Thus, among the vertices $v_1,\ldots,v_l$, the same number belong to $H_h$ and $H_{1-h}$ if $l$ is even, and one more vertex belongs to $H_{1-h}$ than to $H_h$ if $l$ is odd.  We use these facts repeatedly in the proof.
	
	\textbf{Case 1.}  \emph{$k=0$.}  
	If $C$ is all negative, it has even length (because it is balanced) and it is parity signed.
	
	Assume $C$ is not all negative.  Then it has positive sections $P_1, P_2, \ldots, P_r$ and negative sections $M_1,\ldots,M_r$, so $C = M_1P_1M_2P_2 \cdots M_rP_r$.  Let $c_i$ be the initial vertex of $M_i$ and let $c_i'$ be its final vertex, and let $d_i$ be the initial vertex of $P_i$.  By choice of notation, assume $c_1 \in H_1$.  Each $M_i \setminus c_i$ contributes equally many vertices to $H_0$ and $H_1$, since $l(M_i)$ is even.  
	For the same reason, for each $M_i$, $c_i$ and $c_i'$ are in the same part of the Harary bipartition.  Also, for each $i$, all of $V(P_i)$ is contained in the same part of the Harary bipartition.  It follows that $V(P_i) \subseteq H_1$ for all $i$.
	
	Now let us count the vertices in $H_0$ and $H_1$, or more precisely, let us count $|H_1| - |H_0|$.  
	Each $M_i \setminus c_i$ contributes equally many vertices to $H_0$ and $H_1$.  
	Each $P_i \setminus d_i$ contributes $l(P_i)$ vertices to $H_1$.  
	Therefore, $|H_1| - |H_0| = l(P_1) + \cdots + l(P_r)$.  Since this number is non-negative and $C$ is parity signed if and only if $|H_1| - |H_0| = 0, \pm1$, we conclude that $C$ is parity signed if and only if it has at most one positive edge.
	
	Note that if $C$ has no positive edges, since it is balanced it has even length.  
	If $C$ has one positive edge, since it is balanced it has odd length.  
	In both cases it has one negative section, which is even, so it does fall under part (1) of the theorem.
	
	\textbf{Case 2.}  \emph{$k>0$.}  
	We focus attention on the odd negative sections $N_i$.  Let $a_i$ be the initial vertex and $a_i'$ the final vertex of $N_i$.  Let $Q_i$ be the path in $C$ from $a_i'$ to $a_{i+1}$ if $i<k$ and let $Q_k$ be the path from $a_k'$ to $a_i$.  By choice of notation, assume $b_1 \in H_1$.
	
	Because we assumed $a_1' \in H_1$, each $N_i\setminus a_i$ contributes one more vertex to $H_1$ than to $H_0$ if $i$ is odd and one more to $H_0$ than to $H_1$ if $i$ is even.  In total, all $N_i \setminus a_i$ contribute an equal number of vertices to $H_0$ and $H_1$.
	
	Consider a particular $Q_i$.  It consists of positive sections $P_1, \ldots, P_{r_i}$ and even negative sections $M_1,\ldots,M_{r_i-1}$, so that $Q_i = P_1M_1P_2 \cdots M_{r_i-1}P_{r_i}$.  (Possibly $r_i=1$; then $Q_i$ is the positive section $P_1$.)  The initial and final vertices of each even negative section $M_j$ belong to the same part of the Harary bipartition, so if $a_i'$, which is the final vertex of $N_i$ and the initial vertex of $Q_i$, belongs to $H_h$, then $V(P_j) \subseteq H_h$.  It follows that $P_j \setminus b_j$, where $b_j$ is the initial vertex of $P_j$, contributes $l(P_j)$ vertices to $H_h$.  Let $b_j'$ be the initial vertex of $M_j$; then $M_j \setminus b_j'$ contributes the same number of vertices to $H_h$ as to $H_{1-h}$.  Thus, the total contribution of $Q_i$ to $|H_h| - |H_{1-h}|$ is the total length of the positive sections in $Q_i$, which is the number of positive edges in $Q_i$.  This is $m_i$.
	
	Summarizing, $Q_1$ contributes $m_1$ to $|H_1| - |H_0|$, $Q_2$ contributes $m_2$ to $|H_0| - |H_1|$, $Q_3$ contributes $m_3$ to $|H_1| - |H_0|$, and so on.  The total contribution to $|H_1| - |H_0|$ is $m_o - m_e$.  $C$ is parity signed if and only if this number equals $0$, $1$, or $-1$.  That completes the proof.
\end{proof}

\begin{theorem}[Path Theorem]\label{thm:path}
	Let $P$ be a signed path with odd negative sections $N_1, N_2, \ldots, N_k$ ($k\geq0$), in order along $P$.  Let $Q_0$ be the path preceding $N_1$ (possibly of length $0$), $Q_1$ the path between $N_1$ and $N_2$, etc., and $Q_k$ the path following $N_k$ (possibly of length $0$).  Let $m_i$ be the number of positive edges in $Q_i$, and let $m_o$ be the sum of all $m_i$ for odd $i$ and $m_e$ the sum of all $m_i$ for even $i$.  Then $P$ is parity signed if and only if 
	
	$$
	m_e-m_o = \begin{cases}
	1, 0, -1,	&\text{if $k$ is odd}, \\
	0, -1, -2,	&\text{if $k$ is even}.
	\end{cases}
	$$
	
	In particular, when $P$ has no odd negative sections, it is parity signed if and only if it has no positive edges.  When $P$ has exactly one odd negative section, it is parity signed if and only if the numbers of positive edges on the two sides of the odd negative section differ by at most $1$.
\end{theorem}

\begin{proof}
	The proof is similar to that of Theorem \ref{thm:cycle}.  Any path is balanced, so $P$ has a Harary bipartition $V = H_0 \cup H_1$, where we take $H_0$ to contain the initial vertex of $P$.  Let $a_i$ denote the initial vertex of $N_i$ and $a_i'$ the initial vertex of $Q_i$.  As in Theorem \ref{thm:cycle}, $N_i \setminus a_i$ contributes $1$ to $|H_1| - |H_0|$ if $i$ is odd and $-1$ to it if $i$ is even.  Let $m_i$ be the number of positive edges in $Q_i$; then $Q_i \setminus a_i'$ contributes $m_i$ to $|H_1| - |H_0|$ if $i$ is odd and $-m_i$ to it if $i$ is even.  Therefore, $|H_1| - |H_0| = m_o-m_e-1+\delta$, where $\delta = 0$ if $k$ is even and $\delta = 1$ if $k$ is odd.   Since $P$ is parity signed if and only if $|H_1| - |H_0| = -1, 0, 1$, the theorem follows.
\end{proof}

\begin{corollary}\label{corpath}
	A signed path $P_n$, with exactly two sections of opposite parity is a parity signed graph if and only if $n=3$.
\end{corollary}

\begin{theorem}\label{clique}
	If a parity signed graph is negative homogeneous then it is bipartite.
\end{theorem}
\begin{proof}
	This follows from Theorem \ref{thm_balanced}, since a cycle in a negative homogeneous signed graph is positive if and only if it is even.
\end{proof}

\begin{theorem}\label{th_bip}
	A (connected) negative homogeneous signed graph $S$ is a parity signed graph if and only if it is a spanning subgraph of $K_{m,n}$ (all negative) with $|m-n| \leq 1$.
\end{theorem}
\begin{proof}
	Let $V_1(S), V_2(S)$ be the complementary independent subsets of the vertex set of $S$. If $|V_1(S)|=|V_2(S)|$, then to each vertex of one of the partition sets, we assign the odd integers and to each vertex of the other partition set we assign even integers. If $|V_1(S)|=|V_2(S)|+1$, then to the vertices in $V_1(S)$ we assign odd integers and to $V_2(S)$ we assign even integers. In both cases we see that the negative homogeneous bipartite signed graph has a parity signed labelling.
	
	Conversely, assume that $|V_1(S)|>|V_2(S)|+1$; then we cannot label the vertices of each of the partition sets exclusively with either odd or even integers. This forces us to have at least one positive edge in the bipartite signed graph. Hence the theorem follows.
\end{proof}

Even cycles give an infinite family of negative homogeneous parity signed graphs. It is to be noted that this is not the only family of negative homogeneous parity signed graphs; all such signed graphs have been found in Theorem \ref{th_bip}.  We mention some other simple families.
Distributing the odd and even integers is the crucial aspect in constructing a family of negative homogeneous parity signed graphs. Some families of negative homogeneous parity signed graphs are given below.
\begin{itemize}
	\item The ladder graph, $P_n \times K_2$.
	\item The corona of a negative homogeneous parity signed graph $S$ with $\overline{K}_n$, i.e., $S\odot \overline{K}_n$, is negative homogeneous.
\end{itemize}

\begin{theorem}\label{thmstar}
	Let $S=K_{1,m+n}$ be a signed star having $m$ ($n$) positive (negative) edges. Then $K_{1,m+n}$ has a parity labelling if and only if $S$ satisfies any one of the following: {\rm(i)} $n=m$, {\rm(ii)} $n=m+2$, {\rm(iii)} $n=m+1$.
\end{theorem}

\begin{proof}
	We begin by proving necessity.
	Assume that $S$ is a parity signed graph. We show that one of the conditions (i), (ii), or (iii) holds. In $S$, we have $|E(S)|=m+n$ and $|V(S)|=m+n+1$. To label the vertices of $S$ we have $ m+n+1$ integers. Now we have two cases.
	
	\textbf{Case 1}: $m+n+1$ is odd.
	
	It is clear that $m+n$ is even. Let $u$ be a vertex in $S$ such that $d(u)=m+n$. Label $u$ with the integer $m+n+1$, which is an odd integer. Now we are left with $m+n$ integers which have to be assigned to $m+n$ vertices where $m$ positive ($n$ negative) edges are incident to $u$. Per the definition of a parity signed graph, $m$ ($n$) vertices must be labelled with odd (even) integers. As $m+n$ is even, the numbers of odd and even integers are equal. Thus we conclude that $m=n$.  Hence, (i) holds.
	
	Now we label $u$ with the integer $m+n$, which is an even integer. We are left with $m+n$ integers, having $\frac{m+n+2}{2}$ odd integers and $\frac{m+n-2}{2}$ even integers.
	
	Since $u$ has been labelled with an even integer, $m$ vertices must be labelled with even integers and $n$ vertices must be labelled with odd integers. In other words, $\frac{m+n+2}{2}=n$, i.e., $n=m+2$. Similarly, $\frac{m+n-2}{2}=m$, i.e., again $n=m+2$. Thus, (ii) holds.
	
	\textbf{Case 2}: $m+n+1$ is even
	
	It is clear that $m+n$ is odd. Now, we label $u$ with the integer $m+n+1$, which is even. Now, we are left with $m+n$ integers, which is an odd number. Per the definition of a parity signed graph, we have to assign even integers to $m$ vertices and odd integers to $n$ vertices which are adjacent to $u$. Note that we have $m+n$ integers, and $m+n$ is odd. Thus we conclude that $m=\lfloor \frac{m+n}{2} \rfloor$ and $n=\lceil \frac{m+n}{2} \rceil$. In other words, the number of odd integers is one more than the number of even integers. That is, $n=m+1$. Thus, (iii) holds.
	
	Now, we assign the $(m+n)^{\rm th}$ integer to $u$; it is an odd number. Again we are left with $m+n$ integers which have to be assigned to $m+n$ vertices. As $u$ has been assigned an odd integer, $m$ vertices have to be labelled with odd integers and $n$ vertices have to be labelled with even integers. As discussed above, we conclude that $n=m+1$. Hence, (iii) holds. 
	
	Thus necessity is proved.
	
	Sufficiency is obvious.
\end{proof}

Theorem \ref{thmstar} can also be stated as follows.
\begin{theorem}
	Let $K_{(1,~m+n)}$ be a signed star having $m$ positive and $n$ negative edges. Then $K_{(1,~m+n)}$ is a parity signed graph if and only if {\rm(i)} $n=m$ or $n=m+2$ when $m+n$ is even, or {\rm(ii)} $n=m$ or $n=m+1$ when $m+n$ is odd.
\end{theorem}

Do there exist signed bistars which have parity labellings? We answer this question affirmatively and give the structure of signed bistars with parity signed labellings.

\begin{theorem}\label{thm:bistar}
	Let $S:=B^+(m, n)$ be a bistar obtained from a positive edge $uv$ by adding $m$ positive edges and $n$ negative edges to the vertices $u$ and $v$, respectively.  Then $B^+(m, n)$ is a parity signed graph if and only if $n = m+1$ or $m + 3$ when $m+n$ is odd, or $n=m+2$ when $m+n$ is even.
\end{theorem}
\begin{proof}
	For the proof of necessity suppose that $S$ is a parity signed graph. Now, $|E(S)|=m+n+1$ and $|V(S)|=m+n+2$. We thus have $m+n+2$ integers to be assigned to the vertices of $S$. To show that the result holds, we consider two cases.
	
	\textbf{Case 1}. $m+n$ is even.
	
	Clearly, $m+n+2$ is even. Since $uv$ is a positive edge, per the definition of parity signed graph $u$ and $v$ must receive integers of the same parity.
	
	\textit{Subcase 1}. Suppose $u$ and $v$ receive odd integers. Since $m+n+2$ is even, the number of odd integers is equal to the number of even integers, which is $\frac{m+n+2}{2}$. As two odd integers have already been assigned to $u$ and $v$, we are left with $\frac{m+n+2}{2}-2$ odd integers and they must be assigned to $m$ vertices. In other words, we conclude that $\frac{m+n+2}{2}-2=m$. That is, $n=m+2$. Hence, the result holds. 
	
	On the other hand, we have $\frac{m+n+2}{2}$ even integers and they must be assigned to $n$ vertices. Hence, $\frac{m+n+2}{2}=n$. Thus we get $n=m+2$ and again the result holds.
	
	\textit{Subcase 2}. Suppose $u$ and $v$ receive even integers. As discussed above, we have $\frac{m+n+2}{2}-2$ even integers and they must be assigned to $m$ vertices. Thus $\frac{m+n+2}{2}-2=m$, so $n=m+2$ and the result holds.
	
	Further, $\frac{m+n+2}{2}$ odd integers are to be assigned to $n$ vertices. As discussed above, we conclude that $n=m+2$ and the result holds.
	
	\textbf{Case 2}. $m+n$ is odd.
	
	Then $m+n+2$ is odd. Now we have two subcases.
	
	\textit{Subcase 1}. Suppose $u$ and $v$ receive odd integers. From $m+n+2$ (which is an odd integer) we are left $m+n$ integers to be assigned to $m+n$ vertices. Among the $m+n+2$ integers we have $\frac{m+n+3}{2}$ odd integers and $\frac{m+n+1}{2}$ even integers. Among the $\frac{m+n+3}{2}$ odd integers, we have $\frac{m+n-1}{2}$ odd integers to be assigned to $m$ vertices. Thus, $\frac{m+n-1}{2}=m$. That is, $n=m+1$ and the result holds.
	
	On the other hand, we have $\frac{m+n+1}{2}$ even integers to be assigned to $n$ vertices. Hence, $\frac{m+n+1}{2}=n$, and $n=m+1$. Again the result holds.
	
	\textit{Subcase 2}.  Let us assume that $u$ and $v$ have been assigned even integers. Observe that from $m+n+2$ integers, which is an odd number, we are left with $m+n$ integers to be assigned to $m+n$ vertices. Among $m+n+2$ integers, we have $\frac{m+n+3}{2}$ odd integers and $\frac{m+n+1}{2}$ even integers. Among the $\frac{m+n+1}{2}$ even integers, we have $\frac{m+n-3}{2}$ even integers to be assigned to $m$ vertices. Then $\frac{m+n-3}{2}=m$. Thus $n=m+3$ and the result holds.
	
	On the other hand, $\frac{m+n+3}{2}$ odd integers have to be assigned to $n$ vertices. Thus, $\frac{m+n+3}{2}=n$ and we get $n=m+3$. Thus the result holds.
	
	Sufficiency is easy to see.
\end{proof}

Is there another way to view stars or bistars with a parity signed labelling? The answer is yes, as shown in the following results:
\begin{corollary}
	Let $K_{1,n}$ and $K_{1,m}$ be parity signed stars having $u$ and $v$ as their central vertices. A signed bistar obtained from $K_{1,n}$ and $K_{1,m}$ by joining $u$ and $v$ by a positive (negative) edge is a parity signed bistar if and only if the labels of $u$ and $v$ are of the same parity (different parity).
\end{corollary}

\begin{corollary}
	Let $B^{*}(m,n)$ be a negative homogeneous bistar having $m$ and $n$ edges incident to vertices $u$ and $v$ of an edge $uv$, respectively. Then $B^*(m, n)$ is a parity signed signed graph if and only if $n=m$ or $n=m+1$.
\end{corollary}

\begin{corollary}
	A signed $K_{1,n}$ is a parity signed graph if and only if it satisfies the following conditions:
	\begin{enumerate}
		\item[{\rm(a)}] $|E^-(K_{1,n})|=|E^+(K_{1,n})|$, if $n$ is even.
		\item[{\rm(b)}] $|E^-(K_{1,n})|=|E^+(K_{1,n})|+1$, if $n$ is odd.
	\end{enumerate}
\end{corollary}

We have seen in Theorem \ref{thm_balanced}  that every parity signed graph is balanced. But the converse is not true. For example, positive homogeneous signed graphs do not admit parity labellings. Another example is given in Figure \ref{fig:balanced_not_psg}.

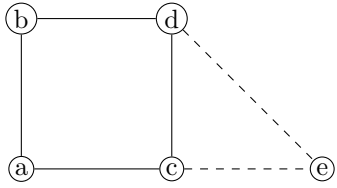
\begin{figure}[h!]\label{fig:balanced_not_psg}
	\centering
	\begin{tikzpicture}
	\vertex (u1) at (0,0) {a};
	\vertex (u3) at (2,0)  {c};
	\vertex (u4) at (2,2)  {d};
	\vertex (u2) at (0,2) {b};	
	\vertex (u5) at (4, 0) {e};
	\path 
	(u1) edge  (u2)
	(u2) edge  (u4)
	(u1) edge  (u3)
	(u3) edge  (u4)
	(u3) edge [dashed] (u5)
	(u4) edge [dashed] (u5)
	;
	\end{tikzpicture}
	\caption{A balanced signed graph but not a parity signed graph.}
\end{figure}

Thus, it is worth characterizing the balanced signed graphs that are parity signed graphs. This characterization is in Theorem \ref{char}.

%\begin{theorem}\label{balancedchar}
%	A (connected) balanced signed graph $S$ is a parity signed graph if and only if its vertex set $V(S)$ can be partitioned  into two subsets $V_1(S)$ and $V_2(S)$ such that negative edges lie across $V_1(S)$ and $V_2(S)$ and $||V_1(S)|-|V_2(S)||\le 1$.
%\end{theorem}
%
%The proof is easy to see.

\section{More on the \textit{rna} Number}

For a graph $G$ of order $n$, there are $n!$ bijective functions $f:V(G) \rightarrow \{1, 2, \dots, n\}$. Each of the bijective functions generates a parity signed graph from $G$. This is where we find the relevance of the \emph{rna} number given in Definition \ref{rna}. 
Previously the \emph{rna} number was treated in \cite{Ach}.  
We now assess the \emph{rna} numbers of some graphs.

\begin{theorem}\label{thm_star}
	For a star $K_{1,n}$, $\sigma^-(K_{1,n})=\lceil{\frac{n}{2}}\rceil$.
\end{theorem} 
\begin{proof}
	The number of negative edges differs depending on the label of the central vertex in a star. Hence, we analyse two cases.
	
	\textbf{Case 1}: $n$ is odd.
	
	Suppose the central vertex is labelled with $n+1$. There are $n$ pendant vertices which are labelled with 1,2,$\dots$,$n$. Hence, there are exactly $\frac{n-1}{2}$ pendant vertices labelled with even integers and $\frac{n+1}{2}$ pendant vertices labelled with odd integers. 
	As $n+1$ is even, there are $\frac{n+1}{2}$ negative edges.
	
	If we swap the labels $n$ and $n+1$, then the central vertex gets the label $n$. Now there are $\frac{n+1}{2}$ pendant vertices labelled with even integers and $\frac{n-1}{2}$ pendant vertices labelled with odd integers. Hence, there will be $\frac{n+1}{2}$ negative edges.
	
	\textbf{Case 2}: $n$ is even.
	
	If the pendant vertices are labelled with $1,2,\dots,n$ and the central vertex is labelled with $n+1$, there are exactly $\frac{n}{2}$ pendant vertices labelled with even integers and $\frac{n}{2}$ pendant vertices labelled with odd integers. As $n+1$ is odd, there are $\frac{n}{2}$ negative edges.
	
	If we swap the labels $n$ and $n+1$, then the central vertex gets the label $n$. Now, there are $\frac{n-2}{2}$ pendant vertices labelled with even integers and $\frac{n+2}{2}$ pendant vertices labelled with odd integers. Hence, there will be $\frac{n+2}{2}$ negative edges.
	
	In both cases $\sigma^-(K_{1,n})=\lceil{\frac{n}{2}}\rceil$.
\end{proof}

\begin{theorem} \label{path_cycle}
	Let $G$ be a path or cycle. The \emph{rna} number of a path of order at least $2$ is $1$.  The \emph{rna} number of a cycle is $2$. Further, $\sigma^-(G)=\sigma^+(G)$ if and only if $G$ is either $P_3$ or $C_4$.
\end{theorem}
\begin{proof}
	It is clear that the \emph{rna} numbers are as stated.  The \emph{adhika} numbers are $\sigma^+(P_n) = \sigma^+(C_n) = n-2$ for a path and cycle of order $n \geq 2$ (path) and $3$ (cycle).  This implies the second half of the theorem.  
\end{proof}

\begin{theorem}
	For a wheel $W_n$, $\sigma^-(W_n) = \lfloor{\frac{n+4}{2}}\rfloor$.
\end{theorem} 
\begin{proof}
	A wheel  $W_n$ is the edge-disjoint union of $K_{1,~n-1}$ and  $C_{n-1}$.  Assume the central vertex has parity $p=0$ or $1$ (even or odd, respectively) and the opposite parity is $1-p$.  Let there be $\alpha$ vertices with parity $1-p$ and $n-\alpha$ with parity $p$.  There are $\alpha$ negative edges in $K_{1,~n-1}$ and a minimum of 2 negative edges in $C_{n-1}$, which is achieved by letting all its vertices with parity $p$ induce a path.  Thus, there are $\alpha+2$ negative edges.
	
	If $n$ is even, $\alpha = \frac{n}{2}$ for a parity labelling.  Thus, there are ${\frac{n+4}{2}}$ negative edges and that is the minimum possible.  Hence, $\sigma^-(W_n) = \lfloor{\frac{n+4}{2}}\rfloor$.
	
	If $n$ is odd, $\alpha = \frac{n-1}{2}$ or $\frac{n+1}{2}$.  The minimum is $\frac{n-1}{2}$, attained by choosing $p=1$ (the centre vertex has an odd label).  In this choice there are $\frac{n+3}{2}$ negative edges, hence, $\sigma^-(W_n) = \lfloor{\frac{n+4}{2}}\rfloor$.
	
	That concludes the proof.
\end{proof}

Are there parity signed graphs with a desired \emph{rna} number? We answer this question in the next theorem.

\begin{theorem}
	For any natural number $k$, there exists a parity signed graph $S$ with $\sigma^-(S)=k$.
\end{theorem}

\begin{proof}
	The star $K_{1,2k}$ of size $2k$ has $\sigma^-(K_{1,2k})=k.$ 
\end{proof}

\begin{theorem}\label{th111}
	Let $G$ be a (connected) graph.  We have $\sigma^-(G)=1$ if and only if $G$ has a cut-edge joining two graphs whose orders differ by at most one.
\end{theorem}
\begin{proof}
	Assume that a graph $G$ has $\sigma^-(G)=1$ and let $S = (G,\sigma)$ be a parity signed graph in which all edges, except one, are positive. This is possible only if the end vertices of positive edges have labels of the same parity. Let the only negative edge have its end vertices $u$ and $v$ labelled $a$ and $b$, respectively. Without loss of generality, assume that $a$ is an odd integer and $b$ is an even integer. All the vertices connected to $u$ without passing through $v$ must have labels with the same parity as $a$ and all the vertices connected to $v$ without passing through $u$ must have labels with the same parity as $b$.  Hence, the edge $uv$ must be a cut-edge.
	
	For the converse, assume that the two components are of equal order.  For any even integer $n$, there exist equal numbers of odd and even integers between 1 and $n$. Hence, the odd integers can be used to label the vertices of one component exclusively and the vertices of the other component can be labelled exclusively with even integers.
	When $n$ is odd, then a similar arrangement will give two components whose orders differ exactly by 1. Clearly, $|E^-(G)|=1$ and this is the smallest possible. Hence, $\sigma^-(G)=1$.
\end{proof}

%\begin{proof}
%	Assume that a parity signed graph $S$ is with $\sigma^-(S)=1$. If $S$ is $P_2$, then the result is obvious. 
%	If $S$ is not $P_2$, then all edges of $S$, except one, are positive. This is possible only if the end vertices of positive edges have labels of the same parity. Let the only negative edge have its end vertices labelled $u$ and $v$. Without loss of generality, assume that $u$ is an odd integer and $v$ is an even integer. All the vertices labelled with odd integers other than $u$ cannot be adjacent to a vertex labelled with an even integer. Similarly, all the vertices labelled with even integers other than $v$ cannot be adjacent to a vertex labelled with an odd integer. Hence, the edge $uv$ must be a cut-edge.
%\end{proof}
%The converse of Theorem \ref{th111} is not true. Nevertheless, we have the following result.
%\begin{theorem}
%	If $S$ is a parity signed graph with a cut-edge joining two graphs whose orders differ by at most one, then $\sigma^-(S)=1$.
%\end{theorem}
%\begin{proof}
%	Assume that the two components are of equal order.  For any even integer $n$, there exist equal numbers of odd and even integers between 1 and $n$. Hence, the odd integers can be used to label the vertices of one component exclusively and the vertices of the other component can be labelled exclusively with even integers.
%	
%	When $n$ is odd, then a similar arrangement will give two components whose orders differ exactly by 1. Clearly, $\sigma^-(S)=1$. 
%\end{proof}

\section{Conclusion}
We have explored the balanced nature of parity signed graphs. We have also given some characterizations of parity signed graphs and the \emph{rna} number $\sigma^-(G)$ of some graphs. We have studied paths, cycles, stars and bistars admitting parity labelling. We have also investigated the effect of the \emph{rna} number on the structure of a signed graph.

For further studies on parity signed graphs, we propose some ideas. Let $S$ be a parity signed graph having a parity labelling $\mu:V(S) \rightarrow \{1, 2, \dots, |V(S)|\}$.
% and let $\overline{S}=(\overline{S}, \sigma)$ be the complement of $S$ with respect to its underlying graph. 
We define the parity complement, denoted as $\overline{S}_p$, of $S$ under the parity labelling $\mu$ as the complement of its underlying graph with the parity signs given by the same labelling $\mu$. Observe that $\overline{S}_p$ will also be a parity signed graph.

\begin{figure}[h!]\label{fig:pscomplement}
	\centering
	\begin{tikzpicture}
	\vertex (u1) at (0,0) {1};
	\vertex (u3) at (2,0)  {3};
	\vertex (u4) at (2,2)  {4};
	\vertex (u2) at (0,2) {2};	
	%Complement
	\vertex (u11) at (4,0) {1};
	\vertex (u33) at (6,0)  {3};
	\vertex (u44) at (6,2)  {4};
	\vertex (u22) at (4,2) {2};
	\path 
	(u1) edge [dashed] (u2)
	(u2) edge  (u4)
	(u3) edge [dashed] (u4)
	(u11) edge (u33)
	(u11) edge [dashed] (u44)
	(u33) edge [dashed] (u22)
	;
	\end{tikzpicture}
	\caption{A parity signed graph and its parity complement.}\label{img2}
\end{figure}
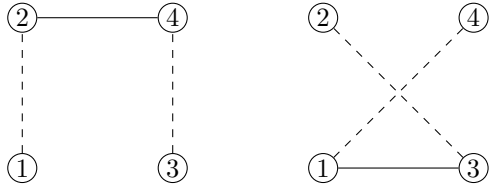

Some problems are:
\begin{enumerate}
	\item Characterize signed graphs whose line signed graphs are parity signed graphs.
	\item Characterize signed graphs $S$ such that  $\sigma^-(\overline{S}_p) = |E^-(\overline{S}_p)|$.
	\item What is the relation between $\sigma^-(S)+\sigma^-(\overline{S}_p)$ and $\sigma^-(S\cup \overline{S}_p)$?
\end{enumerate}

We now define cordiality in parity signed graphs. A parity signed graph $S$ is cordial if $||E^-(S)|-|E^+(S)||\le 1$. A parity signed graph $S$ is absolutely cordial if $|\sigma^-(S)-\sigma^+(S)|\le 1$.
The following problem is worth exploring. 

\textit{Characterize parity signed graphs that are absolutely cordial}.

\section*{Acknowledgment}
We thank all the participants of the Monthly Informal Group Discussion of Bengaluru conducted on the third Sunday of every month, mostly in Christ University, Bengaluru, for their continuous interaction and active participation in the Discrete Mathematics discussions.

\end{document}